\documentclass[conference]{IEEEtran}
\IEEEoverridecommandlockouts
\usepackage{multirow}
\usepackage{mathtools}
\usepackage{eurosym}
\usepackage{cite}
\usepackage{amsmath,amssymb,amsthm,amsfonts}
\usepackage{graphicx}
\usepackage{textcomp}
\usepackage{xcolor}
\usepackage{algorithm, algpseudocode}
\usepackage{setspace}
\theoremstyle{definition}
\newtheorem{proposition}{Proposition}

\def\BibTeX{{\rm B\kern-.05em{\sc i\kern-.025em b}\kern-.08em
    T\kern-.1667em\lower.7ex\hbox{E}\kern-.125emX}}
\begin{document}

\title{Optimal Virtual Power Plant Investment Planning via Time Series Aggregation with Bounded Error\\
\thanks{Funded by the European Union (ERC, NetZero-Opt, 101116212).}
}

\author{\IEEEauthorblockN{Luca Santosuosso}
\IEEEauthorblockA{\textit{Institute of Electricity Economics and Energy Innovation} \\
\textit{Graz University of Technology}\\
Graz, Austria \\
luca.santosuosso@tugraz.at}
\and
\IEEEauthorblockN{Sonja Wogrin}
\IEEEauthorblockA{\textit{Institute of Electricity Economics and Energy Innovation} \\
\textit{Graz University of Technology}\\
Graz, Austria \\
wogrin@tugraz.at}
}

\maketitle

\begin{abstract}
This study addresses the investment planning problem of a virtual power plant (VPP), formulated as a mixed-integer linear programming (MILP) model.
As the number of binary variables increases and the investment time horizon extends, the problem can become computationally intractable.
To mitigate this issue, time series aggregation (TSA) methods are commonly employed.
However, since TSA typically results in a loss of accuracy, it is standard practice to derive bounds to control the associated error.
Existing methods validate these bounds only in the linear case, and when applied to MILP models, they often yield heuristics that may even produce infeasible solutions.
To bridge this gap, we propose an iterative TSA method for solving the VPP investment planning problem formulated as a MILP model, while ensuring a bounded error in the objective function.
Our main theoretical contribution is to formally demonstrate that the derived bounds remain valid at each iteration.
Notably, the proposed method consistently guarantees feasible solutions throughout the iterative process.
Numerical results show that the proposed TSA method achieves superior computational efficiency compared to standard full-scale optimization. \end{abstract}

\begin{IEEEkeywords}
Time series aggregation, investment, mixed-integer linear programming, computational efficiency, bounds.
\end{IEEEkeywords}

\section{Introduction}
Investment planning is a fundamental problem for virtual power plants (VPPs),
aimed at optimizing the mix and sizing of energy resources while minimizing capital investment and operational costs \cite{BARINGO2023101105}.
The large-scale integration of renewable energy sources, the emergence of new flexibility technologies, and market liberalization
have significantly increased the complexity of investment optimization models,
often resulting in large-scale, strongly NP-hard mixed-integer linear programming (MILP) formulations \cite{GODERBAUER2019343}.
A key challenge lies in balancing modeling accuracy with tractability \cite{10863601}.

To address this trade-off, time series aggregation (TSA) methods are commonly employed \cite{8369128}.
These methods reduce input time series to a smaller set of representative periods,
enabling the formulation of an aggregated optimization model that approximates
the solution of the original, full-scale optimization model while reducing computational burden.

TSA is generally classified into \textit{a-priori} and \textit{a-posteriori} methods \cite{hoffmann2020review}.
A-priori methods typically use clustering techniques to identify representative time periods based on the statistical characteristics of the input data.
Common approaches include k-means \cite{10407974}, k-medoids \cite{SCHUTZ2018570}, and hierarchical clustering \cite{8017598}, among others. These methods have been extensively applied to power system investment planning \cite{9543162}, with the prevailing consensus that their effectiveness must be assessed for each specific application \cite{KOTZUR2018474}.
In contrast, a-posteriori methods employ iterative optimization-based procedures,
aiming to minimize the deviation between aggregated and full-scale optimization model solutions, i.e., the output error \cite{8610317}.

A modeler is generally more concerned with minimizing the output error than with accurately representing the input space of a problem.
This has sparked increasing interest in a-posteriori methods \cite{7527691}.
Notably, the analysis in \cite{10037240} demonstrates the potential for tremendous dimensionality reduction in power systems optimization
(from 8760 hours to just 3 representative hours for a linear economic dispatch model)
while achieving zero output error through a well-designed a-posteriori TSA.
However, in practical applications, due to the typically highly nonlinear relationship between the input space and the optimization outcome,
selecting representative periods that capture the critical input features necessary to recover the true optimum of the full-scale model is often a non-trivial task.

To retain the simplicity of a-priori TSA while leveraging the enhanced performance of a-posteriori methods,
a growing trend involves integrating input data clustering within optimization-based TSA algorithms to bound the optimal objective function value of the full-scale optimization model \cite{BAUMGARTNER2019127}.
This methodology has been applied to energy system synthesis optimization \cite{bahl2018typical}, TSA for optimization models with storage while maintaining temporal chronology \cite{TSO2020115190}, and the optimal design of energy supply systems \cite{YOKOYAMA2021120505}.
Formal theoretical results are provided in \cite{TEICHGRAEBER20191283} to validate the methodology employed for deriving these bounds in linear programs with a specific structure.
However, for MILP models, these methods lack formal performance guarantees and often require additional steps to ensure feasibility \cite{BAHL2017900}.

A TSA method for VPP investment optimization that guarantees bounded output error while maintaining feasibility, regardless of the clustering technique employed, is currently lacking.
This study seeks to bridge this gap.

The main contributions of the paper are as follows:
\begin{itemize}
    \item We frame the VPP investment planning problem as a MILP model and demonstrate that a lower bound on the optimal objective function value of the full-scale investment model can be derived from an appropriately designed aggregated model.
    
    \item Building on this result, we develop an algorithm to iteratively derive upper and lower bounds on the optimal objective function value of the full-scale investment model via TSA, with these bounds proven to remain valid regardless of the clustering technique employed. Additionally, the algorithm guarantees that the solution consistently remains feasible within the original feasible region of the full-scale investment model at each iteration.
    
    \item Finally, we evaluate the performance of the proposed algorithm with different clustering techniques.
\end{itemize}

The performance of the proposed algorithm is validated through a case study on the optimal selection and sizing of technologies for a VPP \cite{10267715}.

The remainder of the article is organized as follows: Section~\ref{sec:Methodology} outlines the problem and presents the proposed TSA algorithm with bounded output error; Section~\ref{sec:Results} discusses the results; and Section~\ref{sec:Conclusion} concludes the study.

\section{Methodology}\label{sec:Methodology}
This section outlines the proposed methodology.
We begin with the problem statement in Subsection~\ref{sec:Methodology_problem_statement},
followed by the formulation of the full-scale and the aggregated optimization models in Subsections \ref{sec:Methodology_full_scale_model} and \ref{sec:Methodology_aggregated_model}, respectively.
Our main theoretical result is presented in Subsection~\ref{sec:Methodology_theoretical_result},
which is employed to bound the optimal objective value of the full-scale model using TSA,
as described in Subsection~\ref{sec:Methodology_algorithm}.

We use bold lowercase symbols and bold uppercase symbols to denote sets of decision variables and parameters, respectively. The notation $|\cdot|$ denotes the cardinality of a set.

\subsection{Problem Statement}\label{sec:Methodology_problem_statement}
The goal is to determine the optimal mix and sizing of power generation units that minimizes both capital investment and operational costs, while meeting a desired energy demand.
This problem is relevant, for instance, to a VPP whose primary objective is to balance power generation and consumption while minimizing costs \cite{BARINGO2023101105}.
Henceforth, we shall refer to this problem as the VPP investment planning problem.

It is noteworthy that the proposed formulation excludes grid constraints, implicitly assuming the VPP does not partake in the grid management.
An extension to incorporate grid constraints will be explored in future work.
Additionally, the investment problem is limited to power generation units, while flexibility technologies (e.g., energy storage systems) are disregarded due to the intertemporal dynamics they involve,
which require further procedures to enforce temporal chronology in the TSA method \cite{CARDONAVASQUEZ2024110267}, to be addressed in future studies.

In the following, let $\boldsymbol{G}$ denote the set of generators, indexed by $g$, and $\boldsymbol{T}$ the set of time periods, indexed by $t$, in the VPP investment planning problem.
Let $C^\mathrm{op}_g$ represent the operational cost of energy generation (\euro/MWh) for the $g$-th generator,
and $C^\mathrm{ns}$ the penalty cost for non-supplied energy demand (\euro/MWh).
Additionally, let $C^\mathrm{inv}_g$ denote the capital cost of capacity expansion (\euro/MW) for the $g$-th generator.
The input time series for the VPP investment planning problem consist of the capacity factors of the generators, denoted by $F_{g,t}$ for generator $g$ at time period $t$,
and the energy demand (MWh) at time period $t$, denoted by $D_t$.

\subsection{The Full-Scale Optimization Model}\label{sec:Methodology_full_scale_model}
This subsection presents the formulation of the full-scale optimization model for the VPP investment planning problem.

Let $x_g$ denote the installed capacity (MW) of the $g$-th generator, and $p_{g,t}$ its power generation (MW) at time period $t$.
Additionally, let $d^\mathrm{ns}_t$ denote the non-supplied energy demand (MWh) in the same period.
The binary variable $b_g \in \{0, 1\}$ is employed to ensure that the installed capacity of the $g$-th generator is either $0$ or within the range $[ \underline{X}_g, \overline{X}_g ]$, where $\underline{X}_g$ and $\overline{X}_g$ are the minimum and maximum allowable installed capacities (MW) for the $g$-th generator, respectively.
We formulate the VPP investment planning as a discrete-time optimization problem over $\boldsymbol{T}$, with sampling time $\Delta$ (hours).

For compactness, we group the decision variables of the full-scale optimization model in the set $\boldsymbol{z}$, as follows:
\begin{equation*}
    \boldsymbol{z} \coloneqq \left\{x_g, b_g, p_{g,t}, d^\mathrm{ns}_t \, | \, g \in \boldsymbol{G}, t \in \boldsymbol{T} \right\}.
\end{equation*}

The objective function of the problem is defined as follows:
\begin{equation}\label{full_scale_investment_model_obj}
    J(\boldsymbol{z}) \coloneqq \sum_{g \in \boldsymbol{G}} C^\mathrm{inv}_g x_g + \sum_{t \in \boldsymbol{T}} \left(\sum_{g \in \boldsymbol{G}} C^\mathrm{op}_g p_{g,t} \Delta + C^\mathrm{ns} d^\mathrm{ns}_t\right),
\end{equation}
which represents the sum of both capital investment and operational costs over the planning horizon $\boldsymbol{T}$.

The \textbf{full-scale optimization model} for the VPP investment planning is formulated as the following MILP model:
\begin{subequations}\label{full_scale_investment_model}
\begin{align}
\min_{\boldsymbol{z}} \quad & J\left(\boldsymbol{z}\right)\\
\textrm{s.t.} \quad & \sum_{g \in \boldsymbol{G}} p_{g, t} \Delta + d^\mathrm{ns}_t = D_t, \quad \forall t, \label{full_scale_investment_model_power_balance}\\
& 0 \leq p_{g,t} \leq F_{g,t} \, x_g, \quad \forall g, \forall t, \label{full_scale_investment_model_gen_limits}\\
& b_g \, \underline{X}_g \leq x_{g} \leq b_g \, \overline{X}_g, \quad \forall g. \label{full_scale_investment_model_inv_limits}
\end{align}
\end{subequations}

In \eqref{full_scale_investment_model}, the power balance within the VPP is ensured by the constraints \eqref{full_scale_investment_model_power_balance}, while the constraints \eqref{full_scale_investment_model_gen_limits} and \eqref{full_scale_investment_model_inv_limits} enforce limits on the power generation and the installed capacity for each generator, respectively.

\subsection{The Aggregated Optimization Model}\label{sec:Methodology_aggregated_model}
This subsection presents the formulation of the aggregated optimization model for the VPP investment planning problem.

Mixed-integer investment planning problems are generally strongly NP-hard \cite{GODERBAUER2019343}.
As the number of binaries and time periods grows, even the stylized optimization model \eqref{full_scale_investment_model} can become computationally demanding or intractable.
To mitigate this, TSA can be used to develop an aggregated version of the full-scale model \eqref{full_scale_investment_model},
solved over a reduced set of representative time periods (or clusters), denoted by $\boldsymbol{K}$ and indexed by $k$.
If $|\boldsymbol{K}| \ll |\boldsymbol{T}|$, solving the aggregated model offers a significant computational advantage over its full-scale counterpart.

Let $\boldsymbol{T}_k$ denote the set of time periods $t \in \boldsymbol{T}$ assigned to the $k$-th cluster via TSA. We define $K := |\boldsymbol{K}|$ and $T_k := |\boldsymbol{T}_k|$.

The aggregated model's decision variables are defined as:
\begin{align}
    \hat{x}_g & \coloneqq x_g, \quad \forall g, \label{full_aggregated_variables_relationship1}\\
    \hat{b}_g & \coloneqq b_g, \quad \forall g, \label{full_aggregated_variables_relationship2}\\
    \hat{p}_{g,k} & \coloneqq \frac{1}{T_k} \sum_{t \in \boldsymbol{T}_k} p_{g,t}, \quad \forall g, \forall k, \label{full_aggregated_variables_relationship3}\\
    \hat{d}^{\mathrm{ns}}_k & \coloneqq \frac{1}{T_k} \sum_{t \in \boldsymbol{T}_k} d_t^{\mathrm{ns}}, \quad \forall k. \label{full_aggregated_variables_relationship4}
\end{align}

For compactness, we group the decision variables \eqref{full_aggregated_variables_relationship1}-\eqref{full_aggregated_variables_relationship4} into the set $\boldsymbol{\hat{z}}$, defined as follows:
\begin{equation*}
    \boldsymbol{\hat{z}} \coloneqq \left\{\hat{x}_g, \hat{b}_g, \hat{p}_{g,k}, \hat{d}^\mathrm{ns}_k \, | \, g \in \boldsymbol{G}, \, k \in \boldsymbol{K}\right\}.
\end{equation*}

The objective function $\hat{J}(\boldsymbol{\hat{z}})$ of the aggregated model is defined to approximate the original objective in \eqref{full_scale_investment_model_obj}, now expressed over $K$ representative time periods:
\begin{align}\label{aggregated_investment_model_obj}
    \hat{J}(\boldsymbol{\hat{z}}) \coloneqq \, & \sum_{g \in \boldsymbol{G}} C^{\mathrm{inv}}_g \hat{x}_g \nonumber\\
    & + \sum_{k \in \boldsymbol{K}} T_k \left(\sum_{g \in \boldsymbol{G}} C^\mathrm{op}_g \hat{p}_{g,k} \Delta + C^\mathrm{ns} \hat{d}^\mathrm{ns}_k\right).
\end{align}

The \textbf{aggregated optimization model} is formulated as the following MILP model:
\begin{subequations}\label{aggregated_investment_model}
\begin{align}
\min_{\boldsymbol{\hat{z}}} \quad & \hat{J}\left(\boldsymbol{\hat{z}}\right)\\
\textrm{s.t.} \quad & \sum_{g \in \boldsymbol{G}} \hat{p}_{g,k} \Delta + \hat{d}^\mathrm{ns}_{k} = \frac{1}{T_k} \sum_{t \in \boldsymbol{T}_k} D_t, \quad \forall k, \label{aggregated_investment_model_power_balance}\\
& 0 \leq \hat{p}_{g,k} \leq \frac{\hat{x}_g}{T_k} \sum_{t \in \boldsymbol{T}_k} F_{g,t}, \quad \forall g, \forall k, \label{aggregated_investment_model_gen_limits}\\
& \hat{b}_g \, \underline{X}_g \leq \hat{x}_{g} \leq \hat{b}_g \, \overline{X}_g, \quad \forall g. \label{aggregated_investment_model_inv_limits}
\end{align}
\end{subequations}

Similar to the full-scale model \eqref{full_scale_investment_model}, the power balance within the VPP is ensured by the constraints \eqref{aggregated_investment_model_power_balance}, while the constraints \eqref{aggregated_investment_model_gen_limits} and \eqref{aggregated_investment_model_inv_limits} enforce limits on the power generation and the installed capacity for each generator, respectively.

\subsection{Main Theoretical Result}\label{sec:Methodology_theoretical_result}
This subsection presents our main theoretical result.

\begin{proposition}\label{prop:main_result}
Let $\boldsymbol{z}$ be a feasible solution to the full-scale model \eqref{full_scale_investment_model}.
Let $\boldsymbol{\hat{z}}$ be derived from $\boldsymbol{z}$ accordingly to \eqref{full_aggregated_variables_relationship1}--\eqref{full_aggregated_variables_relationship4}. Then, $\boldsymbol{\hat{z}}$ is a feasible solution to the aggregated model \eqref{aggregated_investment_model} and it holds that
\begin{equation*}
    J\left(\boldsymbol{z}\right) = \hat{J}\left(\boldsymbol{\hat{z}}\right).
\end{equation*}
\end{proposition}

\begin{proof}
Using \eqref{full_aggregated_variables_relationship3} and \eqref{full_aggregated_variables_relationship4},
the power balance constraints \eqref{aggregated_investment_model_power_balance} in the aggregated model are equivalently reformulated as:
\begin{equation}\label{aggregated_investment_model_power_balance_proof}
    \sum_{t \in \boldsymbol{T}_k} \left( \, \sum_{g \in \boldsymbol{G}} p_{g,t} \Delta + d^{\mathrm{ns}}_{t} \right) = \sum_{t \in \boldsymbol{T}_k} D_t, \quad \forall k.
\end{equation}

Using \eqref{full_aggregated_variables_relationship1} and \eqref{full_aggregated_variables_relationship3},
the power generation constraints \eqref{aggregated_investment_model_gen_limits} in the aggregated model are equivalently reformulated as:
\begin{equation}\label{aggregated_investment_model_gen_limits_proof}
    0 \leq \sum_{t \in \boldsymbol{T}_k} p_{g, t} \leq x_g \sum_{t \in \boldsymbol{T}_k} F_{g, t}, \quad \forall g, \forall k.
\end{equation}

Clearly, the individual constraints \eqref{full_scale_investment_model_power_balance} and \eqref{full_scale_investment_model_gen_limits} imply the aggregate constraints \eqref{aggregated_investment_model_power_balance_proof} and \eqref{aggregated_investment_model_gen_limits_proof}, respectively.
Moreover, from \eqref{full_aggregated_variables_relationship1} and \eqref{full_aggregated_variables_relationship2}, the constraints \eqref{full_scale_investment_model_inv_limits} of the full-scale model are equivalent to the constraints \eqref{aggregated_investment_model_inv_limits} of the aggregated model.
Therefore, $\boldsymbol{\hat{z}}$ is a feasible solution to the aggregated model \eqref{aggregated_investment_model}.

%However, the converse does not hold: satisfying the aggregate constraints \eqref{aggregated_investment_model_power_balance_proof} and \eqref{aggregated_investment_model_gen_limits_proof}, whether as equalities or inequalities, does not necessarily imply that the individual constraints \eqref{full_scale_investment_model_power_balance} and \eqref{full_scale_investment_model_gen_limits} hold for all $t \in \boldsymbol{T}$.

Furthermore, using \eqref{full_aggregated_variables_relationship1}, \eqref{full_aggregated_variables_relationship3} and \eqref{full_aggregated_variables_relationship4} in conjunction with the definition of the objective function of the aggregated model \eqref{aggregated_investment_model}, i.e., $\hat{J}(\boldsymbol{\hat{z}})$ as defined in \eqref{aggregated_investment_model_obj}, we obtain:
\begin{align*}
    \hat{J}(\boldsymbol{\hat{z}}) = & \, \sum_{g \in \boldsymbol{G}} C^{\mathrm{inv}}_g x_g + \sum_{k \in \boldsymbol{K}} \Bigg( \sum_{g \in \boldsymbol{G}} C^\mathrm{op}_g \frac{1}{T_k} \sum_{t \in \boldsymbol{T}_k} p_{g,t} T_k \Delta \\
    & \, + C^\mathrm{ns} \frac{1}{T_k} \sum_{t \in \boldsymbol{T}_k} d_t^{\mathrm{ns}} T_k \Bigg) \\
    = & \, \sum_{g \in \boldsymbol{G}} C^{\mathrm{inv}}_g x_g + \sum_{k \in \boldsymbol{K}} \sum_{t \in \boldsymbol{T}_k} \left(\sum_{g \in \boldsymbol{G}} C^\mathrm{op}_g p_{g,t} \Delta + C^\mathrm{ns} d_t^{\mathrm{ns}} \right) \\
    = & \, \sum_{g \in \boldsymbol{G}} C^{\mathrm{inv}}_g x_g + \sum_{t \in \boldsymbol{T}} \left(\sum_{g \in \boldsymbol{G}} C^\mathrm{op}_g p_{g,t} \Delta + C^\mathrm{ns} d_t^{\mathrm{ns}} \right)\\
    = & \, J(\boldsymbol{z}),
\end{align*}
where $J(\boldsymbol{z})$ is defined as in \eqref{full_scale_investment_model_obj}.
\end{proof}

In words, Proposition \ref{prop:main_result} implies that for any feasible solution $\boldsymbol{z}$ of the full-scale model \eqref{full_scale_investment_model}, there exists a corresponding feasible solution $\boldsymbol{\hat{z}}$ of the aggregated model \eqref{aggregated_investment_model} with the same objective function value.
Let $\boldsymbol{z}^\star$ and $\boldsymbol{\hat{z}}^\star$ denote the optimal solutions to the full-scale and aggregated models, respectively. Then, by Proposition \ref{prop:main_result}, the optimal objective function value of the aggregated model provides a lower bound on that of the full-scale model, i.e., $\hat{J}(\boldsymbol{\hat{z}}^\star) \leq J(\boldsymbol{z}^\star)$.

\subsection{A Time Series Aggregation Method with Bounded Error}\label{sec:Methodology_algorithm}
Building upon Proposition \ref{prop:main_result}, in this subsection we develop a TSA method with bounded error in the objective function.

A lower bound for the optimal objective function value of the full-scale model \eqref{full_scale_investment_model} is derived from the aggregated model \eqref{aggregated_investment_model} accordingly to Proposition \ref{prop:main_result}.
Moreover, the full-scale model \eqref{full_scale_investment_model} is a two-stage optimization problem, where the \textit{here-and-now} decisions are the investments and the binary variables, and the \textit{wait-and-see} decisions are the power generation and the unmet demand. 
Thus, fixing the investment variables from the aggregated model \eqref{aggregated_investment_model} in the full-scale model \eqref{full_scale_investment_model}
results in the \textbf{operational optimization} problem, which can be solved in parallel for each time period $t \in \boldsymbol{T}$ to compute an upper bound on the full-scale model’s optimal objective.

This procedure can be performed iteratively to refine the derived bounds by increasing the number of clusters in the aggregated model \eqref{aggregated_investment_model},
yielding the proposed TSA with bounded error in the objective function, as detailed in Algorithm~\ref{alg:TSA_algorithm}.

\begin{algorithm}[t]
\caption{Time Series Aggregation with Bounded Error in the Objective Function}\label{alg:TSA_algorithm}
\begin{algorithmic}[1]
\Require Parameters $\left\{F_{g,t}, D_t, \overline{X}_g, \underline{X}_g \, | \, g \in \boldsymbol{G}, t \in \boldsymbol{T}\right\}$, initial number of clusters $K^0$, step size $\alpha$, optimality threshold $\epsilon^\mathrm{thr}$, and maximum number of iterations $\overline{I}$.

\Ensure Objective function bounds ${{J^\mathrm{UB}}}^\star$ and ${{J^\mathrm{LB}}}^\star$.

\State \textit{Initialization}: $i \gets 0$; $K^i \gets K^0$; $\epsilon^i \gets +\infty$;

\While{$\epsilon^i > \epsilon^{\mathrm{thr}}$ and $i \leq \overline{I}$}

\State Assign the time periods $t \in \boldsymbol{T}$ to $\left\{\boldsymbol{T}_k^i \, | \, k \in \boldsymbol{K}^i\right\}$ using any clustering technique with $K^i$ clusters;

\State $\boldsymbol{\hat{z}}^\star \coloneqq \left\{\hat{x}_g^\star, \hat{b}_g^\star, \hat{p}_{g,k}^\star, \hat{d}^\mathrm{ns \star}_k \, | \, g \in \boldsymbol{G}, \, k \in \boldsymbol{K}\right\} \gets$ Solve the aggregated optimization model \eqref{aggregated_investment_model} for $\left\{\boldsymbol{T}_k^i \, | \, k \in \boldsymbol{K}^i\right\}$;

\State ${\tilde{J}^\mathrm{LB}} \gets \hat{J}\left(\boldsymbol{\hat{z}}^\star\right)$;

\vspace{1pt}

\State $\left\{p_{g,t}^\star, d^\mathrm{ns \star}_t \, | \, g \in \boldsymbol{G} \right\} \gets$ Solve the operational optimization with investment variable values fixed to those in $\boldsymbol{\hat{z}}^\star$; \Comment{In parallel $\forall t$}

\vspace{1pt}

\State ${\tilde{J}^\mathrm{UB}} \gets J\left( \left\{\hat{x}_g^\star, \hat{b}_g^\star, p_{g,t}^\star, d^\mathrm{ns \star}_t \, | \, g \in \boldsymbol{G}, t \in \boldsymbol{T}\right\} \right)$;

\vspace{2pt}

\If{$i = 0$} ${J^\mathrm{LB}}^{i + 1} \gets \tilde{J}^\mathrm{LB}$ and ${J^\mathrm{UB}}^{i + 1} \gets \tilde{J}^\mathrm{UB}$;

\Else

\State ${J^\mathrm{LB}}^{i + 1} \gets \mathrm{max}\left({J^\mathrm{LB}}^{i}, \tilde{J}^\mathrm{LB}\right)$;

\vspace{2pt}

\State ${J^\mathrm{UB}}^{i + 1} \gets \mathrm{min}\left({J^\mathrm{UB}}^{i}, \tilde{J}^\mathrm{UB}\right)$;

\EndIf

\State $\epsilon^{i+1} \gets \text{Evaluate \eqref{optimality_gap}}$ for ${J^\mathrm{LB}}^{i + 1}$ and ${J^\mathrm{UB}}^{i + 1}$;

\State $K^{i+1} \gets K^i + \alpha \, \lfloor \epsilon^{i+1} \rfloor$;

\State $i \gets i+1$;

\EndWhile

\State ${{J^\mathrm{UB}}}^\star \gets {J^\mathrm{UB}}^i$ and ${{J^\mathrm{LB}}}^\star \gets {J^\mathrm{LB}}^i$;

\end{algorithmic}
\end{algorithm}

Let $i$ denote the current iteration of the algorithm, and $\overline{I}$ the maximum number of iterations. The upper and lower bounds for the $i$-th iteration are ${J^{\mathrm{UB}}}^i$ and ${J^{\mathrm{LB}}}^i$, respectively. The algorithm terminates when the optimality gap $\epsilon$, defined as
\begin{equation}\label{optimality_gap}
    \epsilon \coloneqq \frac{{J^{\mathrm{UB}}} - {J^{\mathrm{LB}}}}{{J^{\mathrm{UB}}}},
\end{equation}
falls below a specified threshold $\epsilon^{\mathrm{thr}}$. 

In the first iteration, the number of clusters is initialized to $K^0$.
To reflect the intuition that a large optimality gap should lead to a substantial increase in the number of clusters used in the next iteration,
we propose an adaptive updating rule, where the number of clusters is increased by multiplying a positive integer value $\alpha$ (step size parameter) by the greatest integer less than or equal to $\epsilon$, denoted by $\lfloor \epsilon \rfloor$.

Notably, the validity of the bounds derived via Algorithm~\ref{alg:TSA_algorithm} holds regardless of the clustering technique employed,
and the operational optimization step of the algorithm consistently provides feasible solutions for the full-scale model \eqref{full_scale_investment_model}.

\section{Numerical Results}\label{sec:Results}
This section presents the numerical results obtained by the proposed Algorithm~\ref{alg:TSA_algorithm}. Subsection~\ref{sec:Results1} evaluates the algorithm's performance with various clustering techniques, while Subsection~\ref{sec:Results2} compares it to full-scale optimization.

We assume a non-supplied energy cost of $5,000$ \euro/MWh and investment costs of $40,000$ \euro/MW for thermal generators and $30,000$ \euro/MW for renewable generators.
The operational costs are $50$ \euro/MWh for thermal generators and $3$ \euro/MWh for renewables.
Each generator’s investment is constrained to $[0.1, 1]$ MW.
It is assumed that $20\%$ of the generators are thermal, with the remainder being renewable.
To automate the evaluation process and ensure replicability, we use synthetic data generated as follows. The capacity factor for thermal units is fixed at $F_t = 1$ for all $t \in \boldsymbol{T}$.
For renewable units, the capacity factors are generated as $F_t = e^{Y_t}$ for all $t \in \boldsymbol{T}$, where $Y_t$ follows a normal distribution, i.e., $Y_t \sim \mathcal{N}\left(\mu, \sigma^2\right)$, with mean $\mu = -1$ and standard deviation $\sigma = 0.5$.
We normalize the capacity factors to the range $[0, 1]$.
The energy demand is generated as $D_t \sim \mathcal{U}\left(0, \frac{|\boldsymbol{G}|}{3}\right)$ for all $t \in \boldsymbol{T}$, where $\mathcal{U}(a,b)$ denotes a uniform distribution over the interval $[a, b]$. We set $K^0 = 5$, $\overline{I} = 1000$, $\alpha = 100$, and $\epsilon^\mathrm{thr} = 0.01$ ($1\%$ error).

The optimization models are implemented on an Intel i7 processor and 32 GB of RAM, using Gurobi 12.0.1.

\subsection{Performance Evaluation of the Proposed Time Series Aggregation Method with Various Clustering Techniques}\label{sec:Results1}
\begin{figure}[h]
\centerline{\includegraphics[scale=0.47]{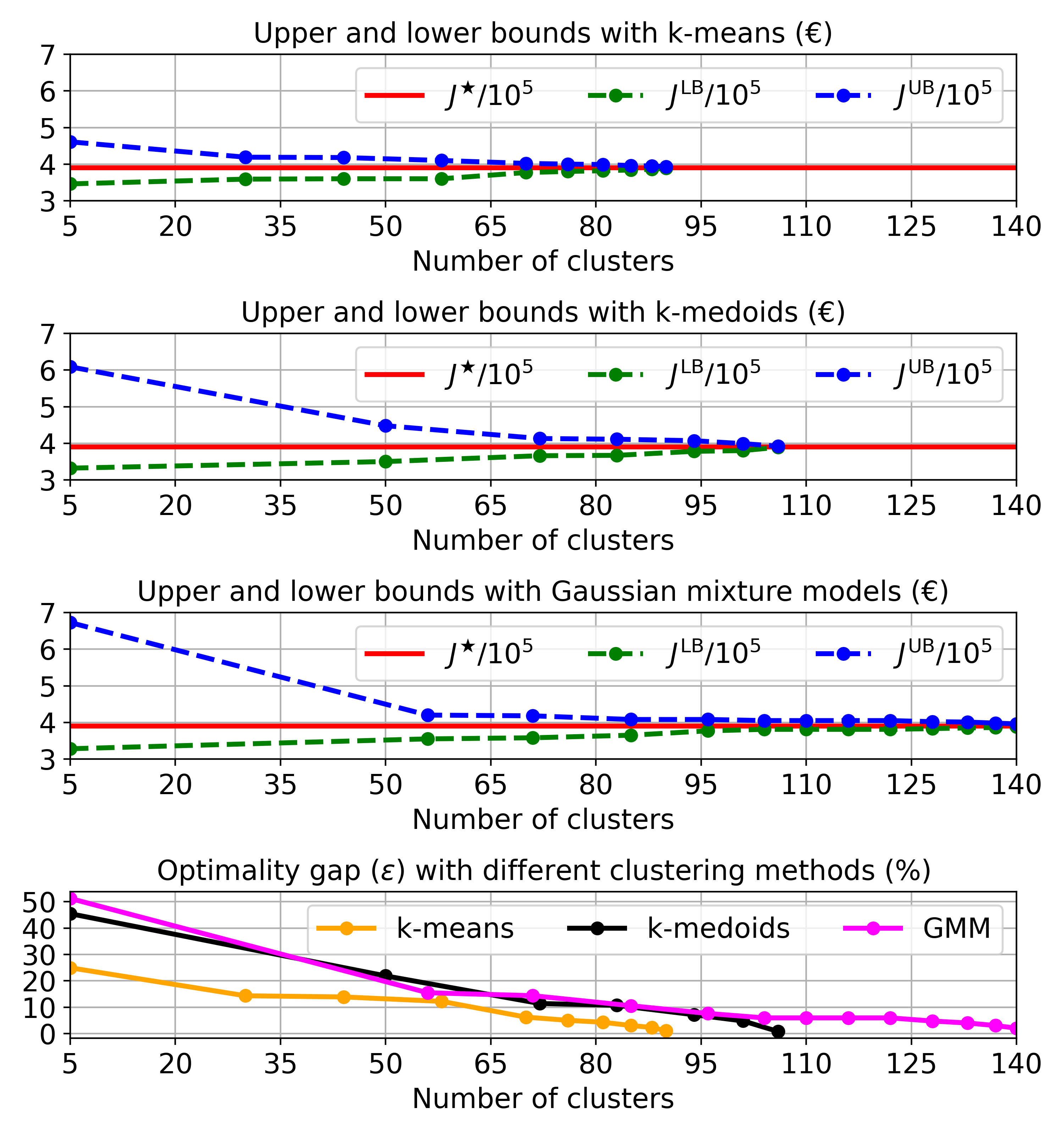}}
\caption{Upper and lower bounds computed for different clustering techniques using the proposed TSA with bounded error in the objective function.}
\label{fig:bounds_and_optim_gap}
\end{figure}

Figure~\ref{fig:bounds_and_optim_gap} presents the results of Algorithm~\ref{alg:TSA_algorithm} for $|\boldsymbol{T}| = 8760$ hours and $|\boldsymbol{G}| = 100$, using k-means, k-medoids, and Gaussian mixture models (GMM) for clustering.
The algorithm achieves less than $1\%$ error with $90$, $106$, and $142$ clusters for k-means, k-medoids, and GMM, respectively.
While k-means requires fewer clusters, k-medoids converges in less iterations ($7$ iterations compared to $10$).
Both k-medoids and GMM show significant initial fluctuations, reflecting their greater sensitivity to complex data structures, while k-means converges more smoothly.
GMM, assuming Gaussian-distributed data, struggles with the non-Gaussian nature of the input data, resulting in repeated bound values before reaching convergence.

\subsection{Comparison with Standard Full-Scale Optimization}\label{sec:Results2}
Table~\ref{tab:TSA_comp_res} presents the computational times of Algorithm~\ref{alg:TSA_algorithm} across various clustering techniques compared to standard full-scale optimization.
Although k-means generally requires more iterations to converge than k-medoids (see Fig.~\ref{fig:bounds_and_optim_gap}),
its lower time per iteration results in a reduced overall computational time.
In contrast, GMM performs the worst, highlighting the challenge of accurately reconstructing the true probability distribution of the input data,
which is critical beyond simply increasing the number of clusters.
For small-scale problems, some clustering techniques may lead to longer computation times than full-scale optimization.
However, for large-scale problems (e.g., $|\boldsymbol{T}| = 17520$ and $|\boldsymbol{G}| = 1000$), the results in Table~\ref{tab:TSA_comp_res} show that the proposed method offers a significant computational advantage over full-scale optimization.

\begin{table}[h]
\caption{Computational times of Algorithm~\ref{alg:TSA_algorithm} with various clustering techniques, compared to full-scale optimization (F-S).
Relative values to F-S are in brackets.
Bold values indicate the best computational performance observed with TSA.}
\begin{center}\label{tab:TSA_comp_res}
\resizebox{8.8cm}{!}{\begin{tabular}{|c|c|c|c|c|}
\hline
\multirow{3}{*}{\textbf{Settings}} & \multicolumn{4}{|c|}{\textbf{Computational time (s)}}\\
\cline{2-5}
& \multirow{2}{*}{\textbf{F-S}} & \multicolumn{3}{|c|}{\textbf{TSA with bounded error}} \\
\cline{3-5} 
& & \textbf{k-means} & \textbf{k-medoids} & \textbf{GMM} \\
\hline
\multicolumn{1}{|l|}{$|\boldsymbol{T}| = 8760$} & \multirow{2}{*}{$273$} & $431$ & $516$ & $689$ \\
\multicolumn{1}{|l|}{$|\boldsymbol{G}| = 100$} & & $(\boldsymbol{+58\%})$ & $(+89\%)$ & $(+152\%)$ \\
\hline
\multicolumn{1}{|l|}{$|\boldsymbol{T}| = 17520$} & \multirow{2}{*}{$793$} & $956$ & $1360$ & $1892$ \\
\multicolumn{1}{|l|}{$|\boldsymbol{G}| = 100$} & & $(\boldsymbol{+21\%})$ & $(+72\%)$ & $(+139\%)$ \\
\hline
\multicolumn{1}{|l|}{$|\boldsymbol{T}| = 8760$} & \multirow{2}{*}{4032} & 3346 & 5113 & 7629 \\
\multicolumn{1}{|l|}{$|\boldsymbol{G}| = 1000$} & & $(\boldsymbol{-17\%})$ & $(+27\%)$ & $(+89\%)$ \\
\hline
\multicolumn{1}{|l|}{$|\boldsymbol{T}| = 17520$} & \multirow{2}{*}{16911} & 7060 & 11862 & 15540 \\
\multicolumn{1}{|l|}{$|\boldsymbol{G}| = 1000$} & & $(\boldsymbol{-58\%})$ & $(-30\%)$ & $(-8\%)$ \\
\hline
\end{tabular}}
\end{center}
\end{table}

\section{Conclusion and Future Work}\label{sec:Conclusion}
This study addresses a VPP investment planning problem formulated as a MILP model.
As the number of binary variables and time periods increases, solving the full-scale model becomes computationally expensive.
To address this, we propose a clustering-based iterative TSA method with bounded error.
The derived bounds are demonstrated to remain valid regardless of the clustering technique used, and feasibility is consistently maintained.
Numerical results show the method's effectiveness and superior computational performance compared to full-scale optimization.
Future work will extend the proposed method to incorporate time-coupling constraints and explore integrating the a-posteriori TSA method from \cite{10037240} to refine the algorithm’s approximations toward exact solutions.

\section{Acknowledgment}
We thank Professor Bettina Klinz for her valuable feedback.

\end{document}